\documentclass[a4paper,leqno]{article}

\usepackage[english]{babel}
\usepackage[T1]{fontenc}
\usepackage[utf8]{inputenc}
\usepackage{amsthm,amssymb,amsmath,amsopn,amsfonts}
\usepackage{enumerate}
\usepackage{csquotes} 
\usepackage{mathtools}

\usepackage[backend=biber,
            style=alphabetic,
            giveninits=true,
            isbn=false,
            doi=false,
            url=false,
            sorting=nyvt]{biblatex}
\renewbibmacro{in:}{}
\DeclareBibliographyDriver{book}{ 
  \printnames{author}
  \newunit\newblock
  \printfield{title}
  \newunit\newblock
  \printlist{publisher}
  \newunit
  \printfield{year}
  \finentry
  }

\newtheorem{thm}{Theorem} 
\newtheorem{thmA}{Theorem} 
\newtheorem{lemma}{Lemma}
\newtheorem{corl}{Corollary}

\renewcommand{\Im}{\operatorname{Im}}
\newcommand{\disk}{\mathbb{D}}
\newcommand{\ddisk}{\partial\mathbb{D}}
\newcommand{\cldisk}{\overline{\mathbb{D}}}
\newcommand{\reals}{\mathbb{R}}
\newcommand{\halfpl}{\mathbb{H}}
\newcommand{\dist}{\operatorname{dist}}

\title{Distortion and Distribution of Sets under Inner Functions}
\author{
        Matteo Levi\thanks{The first author is partially supported by the 2015 PRIN grant \textit{Real and Complex Manifolds: Geometry, Topology and Harmonic Analysis} of the Italian Ministry of Education (MIUR).},
        Artur Nicolau
        and Odí Soler i Gibert\thanks{The three authors are supported in part by the Generalitat de Catalunya (grant 2017 SGR 395) and the Spanish Ministerio de Ciencia e Innovación (projects MTM2014-51824-P, MTM2017-85666-P).}}
\date{}

\newcommand{\Addresses}{{
  \bigskip
  \footnotesize
  
  Matteo~Levi: \textsc{Università di Bologna, Dipartimento di Matematica, Via Zamboni 33, 40126 Bologna, Italia}\par\nopagebreak
  \textit{E-mail address}: \texttt{matteo.levi2@unibo.it}
  
  \medskip

  Artur~Nicolau: \textsc{Universitat Autònoma De Barcelona, Departament de Matemàtiques, Edifici C, 08193-Bellaterra, Catalunya}\par\nopagebreak
  \textit{E-mail address}: \texttt{artur@mat.uab.cat}

  \medskip

  Odí~Soler~i~Gibert: \textsc{Universitat Autònoma De Barcelona, Departament de Matemàtiques, Edifici C, 08193-Bellaterra, Catalunya}\par\nopagebreak
  \textit{E-mail address}: \texttt{odisoler@mat.uab.cat}

}}

\begin{document}

    \maketitle
    
    \begin{abstract}
        It is a classical result that Lebesgue measure on the unit circle is invariant under inner functions fixing the origin.
        In this setting, the distortion of Hausdorff contents has also been studied.
        We present here similar results focusing on inner functions with fixed points on the unit circle.
        In particular, our results yield information not only on the size of preimages of sets under inner functions, but also on their distribution with respect to a given boundary point.
        As an application, we use them to estimate the size of irregular points of inner functions omitting large sets.
        Finally, we also present a natural interpretation of the results in the upper half plane.
    \end{abstract}

    \section{Introduction}

    Let $\disk$ be the open unit disc of the complex plane.
    An analytic mapping $f\colon \disk \rightarrow \disk$ is called \emph{inner} if $ \left|\lim_{r \to 1} f(r\xi)\right| = 1$ for almost every point (a.e.)  $\xi$ of the unit circle $\ddisk$. 
    Hence, an inner function $f$ induces a map defined at almost every point $\xi \in \ddisk$ by $f^\ast(\xi)= \lim_{r \to 1} f(r\xi)$, which we will denote by $f$ as well.
    This induced map lacks the regularity of the inner function itself and it is actually discontinuous at every point $\xi \in \ddisk$ where $f$ does not extend analytically.
    More concretely, fixed $\xi \in \ddisk$ where $f$ does not extend analytically and $\eta \in \ddisk$ there exists a sequence $\xi_n\to \xi$ such that $f(\xi_n) \to \eta$ (see page 77 of \cite{ref:GarnettBoundedAnalyticFunctions}, and page 4 of \cite{ref:NoshiroClusterSets}).
    We are interested in studying certain invariance and distortion properties of measures and Hausdorff contents of sets in the unit circle under the action of inner functions.
  
    Let $f\colon \disk \rightarrow \disk$ be an analytic mapping. We say that a point $p \in \ddisk$ is a \emph{boundary Fatou point} of $f$ if $f(p)=\lim_{r \to 1} f(rp)$ exists and $f(p) \in \ddisk$.
    Hence, the set of boundary Fatou points of an inner function has full measure.
    For $0< \beta < 1$ and $p \in \ddisk$, let $\Gamma_{\beta} (p) = \{z \in \disk : |z-p| < \beta (1-|z|) \}$ be the Stolz angle with opening $\beta$ and vertex at $p$.
    A holomorphic self map $f$ of the unit disc has finite angular derivative at $p \in \ddisk$ if there is a point $\eta \in \ddisk$ and $\beta > 0$ such that the non-tangential limit
    \begin{equation*}
        f'(p) \coloneqq \lim_{\Gamma_{\beta} (p) \ni z \to p} \frac{\eta-f(z)}{p-z}
    \end{equation*}
    exists and is finite. Observe that in this case $ \eta = f(p).$
    We set $|f'(p)|=+\infty$ if the function $f$ does not have a finite angular derivative at the point $p \in \ddisk$. Observe that this is the case if $p$ is not a boundary Fatou point of $f$. With this convention, for any $p \in \ddisk$,  the classical Julia-Carath\'eodory theorem gives  
    \begin{equation}
        \label{eq:ModulusDerivative}
        \liminf_{z \to p} \frac{1-|f(z)|}{1-|z|} = |f'(p)| > 0, 
    \end{equation}
    in the sense that either the $\liminf$ is finite and equal to $|f'(p)|>0$ or both quantities are infinite. See for example Chapters IV and V of \cite{ref:ShapiroCompositionOperators}.

    We denote by $\lambda$ the normalized Lebesgue measure on $\ddisk$ and by $\lambda_z$ the harmonic measure from the point $z \in \disk,$ given by
    \begin{equation*}
        \lambda_z(E) = \int_E \frac{1-|z|^2}{|\xi - z|^2} \, d\lambda(\xi),
    \end{equation*}
    for any measurable set $E \subseteq \ddisk$.
    A classical result due to L\"owner (see, for instance, page 12 of \cite{ref:Ahlfors}) says that Lebesgue measure is invariant under the action of any inner function fixing the origin.
    Hence, the following conformally invariant version of L\"owner's Lemma holds.
    \begin{thmA}
        \label{thm:LownerLemma}
        Let $f:\disk \to \disk$ be an inner function and  $z \in \disk.$ Then,
        \begin{equation*}
            \lambda_z(f^{-1}(E)) = \lambda_{f(z)}(E)
        \end{equation*}
        for any measurable set $E \subseteq \ddisk.$ 
    \end{thmA}

    Observe that, if $z \in \disk$ is a fixed point of $f$, Theorem \ref{thm:LownerLemma} says that $\lambda_z$ is invariant under the action of $f.$
    However, it may be the case that $f$ has no fixed points in $\disk$ but only on $\ddisk.$
    A point $p \in \ddisk$ is a fixed point for $f$ if $\lim_{r \to 1} f(rp) = p.$
    Actually, the classical Denjoy-Wolff Theorem states that for any analytic self mapping $f$ on the unit disc which is not an elliptic automorphism, there exists a fixed point $p \in \cldisk$ of $f,$ called the Denjoy-Wolff fixed point of $f,$ such that the iterates $f^n = f \circ \overset{n)}{\ldots} \circ f$ tend to $p$ uniformly on compacts sets of $\disk.$
    Moreover, $p$ is the unique fixed point of $f$ in $\cldisk$ such that $0 < |f'(p)| \leq 1.$
    See for example Chapter V of \cite{ref:ShapiroCompositionOperators}. 
    We are interested in analogues of Theorem \ref{thm:LownerLemma} when $z \in \ddisk.$
    This situation occurs naturally when the Denjoy-Wolff fixed point of $f$ is on the unit circle.
    In this situation, instead of considering the harmonic measure from a point in the open unit disc, it is natural to measure sets with respect to boundary points.
    We will consider a measure introduced by Doering and Mañé in \cite{ref:DoeringMane}. 
    Fix a point $p \in \cldisk$ and consider the positive measure $\mu_p$ on $\ddisk$ defined by 
    \begin{equation*}
        \mu_p(E) = \int_E \frac{1}{|\xi -p|^2} \, d\lambda(\xi)
    \end{equation*}
    for any measurable set $E \subseteq \ddisk.$
    Observe that for a point $p \in \ddisk$ the measure $\mu_p$ is not finite, while for $p \in \disk,$ it is just a scalar multiple of the harmonic measure given by $\mu_p = (1-|p|^2)^{-1} \lambda_p.$
    A very natural interpretation of the measure $\mu_p$ when $p \in \ddisk$ is the following.
    Let $\omega_p \colon \disk \rightarrow \halfpl$ be the conformal map from the disc into the upper half-plane $\halfpl$ such that $\omega_p(p) = \infty$ and $\omega_p(0) = i/2.$
    Then, for any measurable set $E \subseteq \ddisk,$ we have that $\mu_p(E) = |\omega_p(E)|,$ where we denote by $|A|$ the Lebesgue measure of a set $A \subseteq \reals.$
    Roughly speaking, for a point $p \in \ddisk,$ the measure $\mu_p$ gives information about the size and the distribution of a set around the point $p.$
    Sets having large $\mu_p$ measure are those that are highly concentrated around the point $p.$
    In particular, if $E$ is an open neighbourhood of $p,$ then $\mu_p(E) = \infty.$
    Our first result is the following analogue of Theorem \ref{thm:LownerLemma}.

    \begin{thm}
        \label{thm:LownerMuP}
        Let $f\colon \disk \to \disk$ be an inner function and let $p \in \ddisk$ be a boundary Fatou point of $f.$
        
        \begin{enumerate}[(a)]
            \item
            \label{stm:LownerMuPFiniteDerivative}
            Assume $|f'(p)| < \infty.$
            Then
            \begin{equation*}
                \mu_p(f^{-1}(E)) = |f'(p)|\mu_{f(p)}(E)
            \end{equation*}
            for any measurable set $E \subseteq \ddisk.$
            
            \item
            If $|f'(p)| = \infty$ and $E \subseteq \disk$ is a measurable set, then $\mu_p(f^{-1}(E)) = \infty$ if $\mu_{f(p)}(E) > 0$ and  $\mu_p(f^{-1}(E)) = 0$ if $\mu_{f(p)}(E) = 0.$
        \end{enumerate}
    \end{thm}

    As we can see, we still have a general relation between the measure of a set and its preimage under $f,$ independent from the set.
    Nonetheless, in this case, a distortion term appears and it is given by the size of the angular derivative at the point $p.$
    If $p \in \ddisk$ is the Denjoy-Wolff fixed point of $f,$ this result was previously proved in \cite{ref:DoeringMane}. 
    
    In \cite{ref:FernandezPestanaDistortionInnerFunctions}, Fernández and Pestana studied the distortion of Hausdorff contents under inner functions.
    Fixed $z \in \disk$ and $0 < \alpha < 1,$ consider the Hausdorff content defined as 
    \begin{equation*}
        M_{\alpha}(\lambda_z)(E) = \inf \sum_j \lambda_z (I_j)^{\alpha}, 
    \end{equation*}
    where the infimum is taken over all collections of arcs $\{I_j \}$ of the unit circle such that $E \subseteq \bigcup I_j$. Thus $M_{\alpha}(\lambda_0) (E)$ is the standard Hausdorff content of $E,$ which is denoted by $M_\alpha (E)$. Observe that if $z \in \disk$ and $\tau$ is the automorphism of $\disk$ which interchanges $z$ and $0$, then $M_{\alpha}(\lambda_z)(E) = M_\alpha ({\tau}^{-1} (E))$ for any $E \subseteq \ddisk $. Fern\'andez and Pestana proved the following result, analogous to Theorem \ref{thm:LownerMuP} for Hausdorff contents, stated here in a conformally invariant way.
    \begin{thmA}
        \label{thm:ConformalFernandezPestana}
        For any $0 < \alpha < 1$ there exists a constant $C_\alpha > 0$ such that, if $f\colon \disk \to \disk$ is an inner function and $z \in \disk,$ we have  
        \begin{equation*}
            M_{\alpha}(\lambda_z)(f^{-1}(E)) \geq C_{\alpha} M_{\alpha}(\lambda_{f(z)})(E)
        \end{equation*}
        for any Borel set $E \subseteq \ddisk.$
    \end{thmA}

    It is also shown in \cite{ref:FernandezPestanaDistortionInnerFunctions} that there exists an inner function $f$ such that the preimage of a single point has Hausdorff dimension $1.$
    Hence, the converse estimate in Theorem \ref{thm:ConformalFernandezPestana} is false.
    It is worth mentioning that a related result for sets $E \subseteq \disk$ was established in \cite{ref:Hamilton}.
    For $0 < \alpha < 1$ and $p \in \ddisk,$ we define the $(p,\alpha)$-Hausdorff content of a Borel set $E \subseteq \ddisk$ as
    \begin{equation*}
        M_{\alpha}(\mu_p)(E) \coloneqq \inf \sum_j \mu_p (I_j)^{\alpha}, 
    \end{equation*}
    where the infimum is taken over all collections of arcs $\{I_j\}$ of the unit circle such that $E \setminus \{p\} \subseteq \bigcup I_j.$
    Our second result is the following analogue of Theorem \ref{thm:ConformalFernandezPestana} when $z \in \ddisk.$
    
    \begin{thm}
        \label{thm:ContentsMuP}
        Let $f\colon \disk \to \disk$ be an inner function and let $p \in \ddisk$ be a boundary Fatou point of $f.$

        \begin{enumerate}[(a)]
            \item
            \label{stm:ContentsMuPFiniteDerivative}
            Assume $|f'(p)| < \infty.$
            Then for any $0 < \alpha < 1$ there exists a constant $C_\alpha > 0,$ independent of $f,$ such that
            \begin{equation*}
                M_{\alpha}(\mu_p)(f^{-1}(E)) \geq C_{\alpha} |f'(p)|^{\alpha} M_{\alpha}(\mu_{f(p)})(E)
            \end{equation*}
            for any Borel set $E \subseteq \ddisk.$

            \item
            Assume $|f'(p)| = \infty.$
            Then we have that $M_{\alpha}(\mu_p)(f^{-1}(E)) = \infty$ for any Borel set $E \subseteq \ddisk$ such that  $M_{\alpha}(\mu_{f(p)})(E) > 0.$
        \end{enumerate}
    \end{thm}

    The proofs of Theorem \ref{thm:LownerMuP} and Theorem \ref{thm:ContentsMuP} are given in Section \ref{sec:BoundaryTheorems}.
    In Section \ref{sec:Applications} we give two applications of our results.
    The first one concerns a smoothness property of inner functions which omit large sets of the unit disc and it is inspired on a nice result in \cite{ref:FernandezPestanaDistortionInnerFunctions}.
    In the second application we obtain analogue results on distortion of sets in the real line under inner mappings of the upper half plane. 
    
    It is a pleasure to thank J.~J.~Donaire, J.~L.~Fernández, P.~Gorkin and M.~V.~Melián for helpful discusions.

    \section{Boundary distortion theorems}
    \label{sec:BoundaryTheorems}
    In this section we prove our main results.
    We start with some elementary properties of the measure $\mu_p$ and the content $M_{\alpha}(\mu_p).$
    Recall that a sequence of points $\lbrace p_n \rbrace \subseteq \disk$ converges non-tangentially to a point $p \in \ddisk$ if $\lim p_n = p$ and there exists $\beta > 0$ such that $\{p_n\} \subseteq \Gamma_\beta (p).$

    \begin{lemma}
        \label{lemma:MeasuresConvergence}
        Let $p \in \ddisk.$ For every sequence of points $\lbrace p_n \rbrace \subseteq \disk$ converging non-tangentially to $p,$ we have
        \begin{equation*}
            \mu_{p_n}(E) \longrightarrow \mu_p(E),\quad \text{ as } n \to \infty, 
        \end{equation*}
        for any measurable set $E \subseteq \ddisk.$
    \end{lemma}
    \begin{proof}
        Let $\lbrace p_n \rbrace_n \subseteq \disk$ be any sequence of points approaching $p,$ and write $\mu_n = \mu_{p_n}$ for every $n \geq 1.$
        By Fatou's Lemma, we have
        \begin{equation*}
            \liminf_n \mu_n(E) \geq \int_E \lim_n \frac{1}{|\xi -p_n|^2}\, d\lambda(\xi) = \mu_p(E),
        \end{equation*}
        from which it follows that the result is true when $\mu_p(E) = \infty.$
        So assume $\mu_p(E) < \infty.$
        Fix $\varepsilon > 0$ and consider an arc $I$ centred at $p$ and such that $\mu_p(E \cap I) < \varepsilon.$
        Since $p_n \to p$ non-tangentially, there exists a constant $C > 0$ such that $|\xi - p_n| \geq C |\xi - p|$ for every $\xi \in \ddisk$ and every $n \geq 1.$
        Hence, we have that $\mu_n(E \cap I) \leq C^{-2} \varepsilon$ for every $n.$
        On the other hand, by dominated convergence, we have that
        \begin{equation*}
            \mu_n(E \cap (\ddisk \setminus I)) \longrightarrow \mu_p(E \cap (\ddisk \setminus I)),\quad \text{ as } n \to \infty,
        \end{equation*}
        from which the result follows.
    \end{proof}

    Observe that the assumption on the non-tangential convergence of the sequence $\{p_n\}$ to $p$ only enters into play if $p \in \overline{E}.$
    If $p \notin \overline{E},$ the result holds true for any approaching sequence.
    However, as the following example shows, Lemma \ref{lemma:MeasuresConvergence} fails badly if $p_n$ approaches $p$ tangentially.
    Fix a point $p \in \ddisk$ and consider a sequence of points $\{\xi_n\} \subseteq \ddisk$ such that $|\xi_n - p| = 1/(2n)$ for every $n \geq 1.$
    Consider as well the sequence of pairwise disjoint arcs $\{I_n\}$ such that $I_n$ is centred at $\xi_n$ and $\lambda(I_n) = 1/(4n^4)$ for every $n \geq 1.$
    Now, let $E \coloneqq \bigcup_n I_n,$ $p_n = \left(1-\lambda(I_n)\right) \xi_n,$ and $\mu_n = \mu_{p_n},$ for every $n \geq 1.$
    Since $(1-|p_n|)/|p-p_n| \leq 1/n^3 \longrightarrow 0,$ the sequence $\{p_n\}$ converges to $p$ tangentially.
    For $\xi \in I_n,$ we have $|p_n-\xi| \leq 2 \lambda(I_n)$ and $\mu_n(I_n) \geq (4\lambda(I_n))^{-1} = n^4.$
    Now, on one hand we have $\mu_n(E) \geq \mu_n(I_n) \longrightarrow \infty,$ as $n\to \infty.$
    On the other hand since $|p - \xi| \leq 1/n$ for any $\xi \in I_n$, we have $\mu_p (I_n) \leq n^2 \lambda(I_n) = 1/4n^2$ and we deduce
    \begin{equation*}
        \mu_p(E) = \sum_n \mu_p(I_n) < \infty. 
    \end{equation*}

    For $0 < \alpha < 1$ and $z \in \disk$ consider the  $(z,\alpha)$-Hausdorff content of a Borel set $E \subseteq \ddisk$ defined as
    \begin{equation*}
        M_{\alpha}(\mu_z)(E) =  \inf \sum_j \mu_z(I_j)^{\alpha},
    \end{equation*}
    where the infimum is taken over all collections of arcs $\{I_j\}$ such that $E \subseteq \bigcup I_j.$
    
    \begin{lemma}
        \label{lemma:BoundaryInteriorMuP}
        Given $p \in \ddisk$ and $\beta > 0,$ let $\Gamma_\beta(p)$ be the Stolz angle of opening $\beta$ with vertex at $p.$
        Then there exists a constant $C = C(\beta) > 0$ such that 
        $$
        \mu_z(A) \leq C \mu_p(A)
        $$
        for any measurable set $A \subseteq \ddisk$ and any $z \in \Gamma_\beta(p).$
        Consequently, for any $0 < \alpha < 1$ we also have $M_\alpha(\mu_z)(A) \leq C^{\alpha} M_\alpha(\mu_p)(A)$ for any set $A \subseteq \ddisk $ and any $z \in \Gamma_\beta(p).$
    \end{lemma}
    \begin{proof}
        Observe that there exists a constant $C = C(\beta) > 0$ such that $|\xi - z| \geq C |\xi - p|$ for any $z \in \Gamma_\beta(p)$ and any $\xi \in \ddisk.$
        Hence, $\mu_z(A) \leq C^{-2} \mu_p(A)$ for any measurable set $A \subseteq \ddisk $ and any $z \in \Gamma_\beta(p).$
        This last estimate also gives $M_\alpha(\mu_z)(A) \leq C^{-2 \alpha} M_\alpha(\mu_p)(A).$
    \end{proof}

    The corresponding result to Lemma \ref{lemma:MeasuresConvergence} for Hausdorff contents reads as follows.
    \begin{lemma}
        \label{lemma:ContentsNTConvergence}
        Let  $0 < \alpha < 1$ and $p \in \ddisk.$
        For any sequence of points $\lbrace p_n \rbrace \subseteq \disk$ converging non-tangentially to $p,$ we have
        \begin{equation}
            \label{eq:ContentLimit}
            \lim_{n \to \infty} M_\alpha(\mu_{p_n})(E) = M_\alpha(\mu_p)(E)
        \end{equation}
        for any set $E \subseteq \ddisk.$
    \end{lemma}
    \begin{proof}
        Write $\mu_n = \mu_{p_n}$ for every $n \geq 1.$
        Assume that $M_\alpha(\mu_p)(E) < \infty.$
        In this case, we split the proof of the result into two parts.
        First we show that
        \begin{equation}
            \label{eq:ContentLimitSup}
            \limsup_{n \to \infty} M_\alpha(\mu_{n})(E) \leq M_\alpha(\mu_p)(E),
        \end{equation}
        and then we prove that
        \begin{equation}
            \label{eq:ContentLimitInf}
            \liminf_{n \to \infty} M_\alpha(\mu_{n})(E) \geq M_\alpha(\mu_p)(E),
        \end{equation}
        from which \eqref{eq:ContentLimit} follows immediately.
        To prove \eqref{eq:ContentLimitSup}, given $\varepsilon >0,$ take a covering by open arcs $\{I_j\}$ of the set $E \setminus \{p\}$ such that
        \begin{equation*}
            \sum_j \mu_p(I_j)^\alpha \leq M_\alpha(\mu_p)(E) + \varepsilon.
        \end{equation*}
        Now, by Lemma \ref{lemma:BoundaryInteriorMuP}, for each interval $I_j$ and for every $n \geq 1$ we have that
        \begin{equation*}
            \mu_{n}(I_j) \leq C\mu_p(I_j).
        \end{equation*}
        Thus, by Lemma \ref{lemma:MeasuresConvergence} and dominated convergence, we get that
        \begin{equation*}
            \sum_j \mu_{n}(I_j)^\alpha \longrightarrow \sum_j \mu_p(I_j)^\alpha,\quad \text{ as } n \to \infty.
        \end{equation*}
        By definition, $M_\alpha(\mu_{n})(E) \leq \sum_j \mu_{n}(I_j)^\alpha$ and, thus \eqref{eq:ContentLimitSup} follows immediately.
    
        We prove inequality \eqref{eq:ContentLimitInf} considering two cases.
        Assume first that $p \not\in \overline{E}.$
        Pick $\varepsilon > 0$ and a covering of $E$ by open arcs $\{I_j\},$ such that $\dist(I_j,p) \geq \dist(\overline{E},p)/2$ for every arc $I_j.$
        Observe that, in this situation, there exists $n_0 > 0$ such that if $ n > n_0,$ we have that 
        \begin{equation*}
            \mu_{n}(I_j) \geq (1-\varepsilon)^{1/\alpha} \mu_p(I_j)
        \end{equation*}
        for every arc $I_j$ in our covering.
        Thus, for any such covering of $E  \setminus \{p\},$ if $ n > n_0$ we have that
        \begin{equation*}
            \label{eq:AlphaSumLimitInf}
            \sum_j \mu_{n}(I_j)^\alpha \geq (1-\varepsilon) M_\alpha(\mu_p)(E).
        \end{equation*}
        Observe that the infimum of $\sum_j \mu_{n}(I_j)^\alpha$ when ranging over all coverings $\{I_j\}$ of $E \setminus \{p\}$ by open arcs satisfying that $\dist(I_j,p) \geq \dist(\overline{E},p)/2$ is, precisely, $M_\alpha(\mu_{n})(E).$
        Hence, equation \eqref{eq:ContentLimitInf} follows in the case that $p \not\in \overline{E},$ and therefore equation \eqref{eq:ContentLimit} as well in this situation.
  
        In the case that $p \in \overline{E},$ since we assumed that $M_\alpha(\mu_p) (E) < \infty,$ given $\varepsilon > 0$ we can choose $\delta = \delta(\varepsilon) > 0$ such that
        $M_\alpha(\mu_p)(E \cap I(p,\delta)) < \varepsilon,$  where $I(p,\delta)$ denotes the arc centred at $p$ of length $\delta.$
        Let us denote $E_\delta = E  \setminus I(p,\delta).$
        Since $p \not \in \overline{E_\delta},$ we already know that
        \begin{equation*}
            \lim_{n \to \infty} M_\alpha(\mu_{n})(E_\delta) = M_\alpha(\mu_p)(E_\delta) \geq M_\alpha(\mu_p)(E) - \varepsilon.
        \end{equation*}
        Hence, for any given $\varepsilon > 0,$ we have 
        \begin{equation*}
            \liminf_{n \to \infty} M_\alpha(\mu_{n})(E) \geq \lim_{n \to \infty} M_\alpha(\mu_{n})(E_\delta) \geq M_\alpha(\mu_p)(E) - \varepsilon.
        \end{equation*}
        This concludes the proof whenever $M_\alpha(\mu_p) (E) < \infty.$

        Assume now that $M_\alpha(\mu_p) (E) = \infty.$
        In this case, for any $N > 0$ we can find $\delta = \delta(N) > 0$ such that $M_\alpha(\mu_p)(E_\delta) > N,$ where again $E_\delta = E  \setminus I(p,\delta).$
        Since $p \not \in \overline{E_\delta},$ we have that
        \begin{equation*}
            \lim_{n \to \infty} M_\alpha(\mu_{n})(E_\delta) = M_\alpha(\mu_p)(E_\delta) > N.
        \end{equation*}
        Hence, there exists $n_0 > 0$ such that if $n > n_0,$ then $M_\alpha(\mu_{n})(E_\delta) > N.$
        Using that $M_\alpha(\mu_{n})(E) \geq M_\alpha(\mu_{n})(E_\delta),$ we get \eqref{eq:ContentLimit} in the case in which $M_\alpha(\mu_p) (E) = \infty$ as well.
    \end{proof}
    
    We will use the following auxiliary result which is certainly well known.
    It is included because we have not found a precise reference.

    \begin{lemma}
        \label{lemma:FiniteDerivativeCones}
        Let $f$ be a holomorphic self map of the unit disc.
        Let $\{p_n\}$ be a sequence of points in $\disk$ converging non-tangentially to a point $p \in \ddisk.$
        Assume that $|f'(p)| < \infty,$ then $\{f(p_n)\}$ also converges to $f(p) \in \ddisk$ non-tangentially. 
    \end{lemma}
    \begin{proof}
        Since $|f'(p)| < \infty$ we have that $f(p) \in \ddisk.$
        Write
        \begin{equation*}
            \frac{1-|f(p_n)|}{|f(p)-f(p_n)|} = \frac{1-|f(p_n)|}{1-|p_n|} \frac{1-|p_n|}{|p-p_n|} \frac{|p-p_n|}{|f(p)-f(p_n)|}.
        \end{equation*}
        Also because $|f'(p)|<\infty$, by Julia-Carathéodory Theorem, the first and third terms converge respectively to $|f'(p)|$ and $|f'(p)|^{-1}$, and therefore
        \begin{equation*}
            \liminf_n \frac{1-|f(p_n)|}{|f(p)-f(p_n)|} = \liminf_n \frac{1-|p_n|}{|p-p_n|} > 0.
        \end{equation*}
    \end{proof}

    Note that the assumption of finite angular derivative is necessary in the above statement, even if we ask the function $f$ to be inner.
    In fact, it can be proved that there exist inner functions mapping a given Stolz angle to a tangential region (see \cite{ref:DonaireRadialBehaviourLittleBloch}).

    We are now ready to prove our main results. 

    \begin{proof}[Proof of Theorem \ref{thm:LownerMuP}]
        We can choose a sequence of points $\{p_n\}$ in $\disk$ approaching $p$ non-tangentially such that
        \begin{equation}
            \label{eq:AngularDerivativeCondition}
            \lim_{n \to \infty} \frac{1-|f(p_n)|^2}{1-|p_n|^2} = |f'(p)| > 0.
        \end{equation}

        By Theorem \ref{thm:LownerLemma}, we have that
        \begin{equation}
            \label{eq:LownerPropertyMuPN}
            \mu_{p_n}(f^{-1}(E)) = \frac{1-|f(p_n)|^2}{1-|p_n|^2} \mu_{f(p_n)}(E).
        \end{equation}
  
        Lemma \ref{lemma:MeasuresConvergence} gives that $\mu_{p_n}(f^{-1}(E)) \rightarrow \mu_p(f^{-1}(E))$ as $n \to \infty.$
        If $|f'(p)| < \infty,$ Lemma \ref{lemma:FiniteDerivativeCones} gives that $f(p_n)$ converges to $f(p)$ non-tangentially.
        Thus, Lemma \ref{lemma:MeasuresConvergence} gives that $\mu_{f(p_n)}(E) \rightarrow \mu_{f(p)}(E)$ as $n \to \infty.$
        Therefore, equations \eqref{eq:AngularDerivativeCondition} and \eqref{eq:LownerPropertyMuPN} give the statement \eqref{stm:LownerMuPFiniteDerivative}.
        Assume now that $|f'(p)| = \infty.$
        If $\mu_{f(p)}(E) = 0,$ we have $\lambda(E) = 0.$
        Hence, by Theorem \ref{thm:LownerLemma}, we have that $\lambda(f^{-1}(E)) = 0$ and it follows that $\mu_p(f^{-1}(E)) = 0.$
        Finally assume $\mu_{f(p)}(E) > 0.$
        Observe that for any $n \geq 1$ we have $\mu_{f(p_n)}(E) > \lambda(E) /4 >0.$
        Thus, since $|f'(p)| = \infty,$ the right-hand side of equation \eqref{eq:LownerPropertyMuPN} tends to infinity and, by Lemma \ref{lemma:MeasuresConvergence}, we deduce that $\mu_p (f^{-1}(E)) = \infty.$
    \end{proof}

    \begin{proof}[Proof of Theorem \ref{thm:ContentsMuP}]
        We will use Theorem \ref{thm:ConformalFernandezPestana} in the following form.
        For $z \in \disk$  we have that
        \begin{equation}
            \label{eq:GeneralisedInternalFP}
            M_\alpha(\mu_z)(f^{-1}(E)) \geq C_\alpha \left(\frac{1-|f(z)|^2}{1-|z|^2}\right)^\alpha M_\alpha(\mu_{f(z)})(E)
        \end{equation}
        for any Borel set $E \subseteq \ddisk.$ We can choose a sequence of points $\{p_n\}$ in $\disk$ approaching $p$ non-tangentially such that
        \begin{equation}
            \label{eq:AngularDerivativeCondition2}
            \lim_{n \to \infty} \frac{1-|f(p_n)|^2}{1-|p_n|^2} = |f'(p)| > 0.
        \end{equation}
        Assume $|f'(p)|< \infty.$
        Applying Lemma \ref{lemma:ContentsNTConvergence} and equation \eqref{eq:GeneralisedInternalFP}, we get
        \begin{equation*}
            \begin{split}
                M_{\alpha}(\mu_p)(f^{-1}(E)) &= \lim_{r\to 1} M_{\alpha}(\mu_{p_n})(f^{-1}(E))\\
                &\geq \limsup_{n\to \infty} C_{\alpha} \left(\frac{1-|f(p_n)|^2}{1-|p_n|^2}\right)^{\alpha} M_{\alpha}(\mu_{f(p_n)})(E)\\
                &= C_{\alpha} |f'(p)|^{\alpha} \limsup_{n\to \infty} M_{\alpha}(\mu_{f(p_n)})(E).
            \end{split}
        \end{equation*}
        By Lemma \ref{lemma:FiniteDerivativeCones}, $f(p_n)$ tends to $f(p)$ non-tangentially as $n \to \infty$ and hence, Lemma \ref{lemma:ContentsNTConvergence} gives that
        \begin{equation*}
            \lim_{n \to \infty} M_{\alpha}(\mu_{f(p_n)})(E) =  M_{\alpha}(\mu_{f(p)})(E),
        \end{equation*}
        which finishes the proof of part \eqref{stm:ContentsMuPFiniteDerivative}.
        Assume now $|f'(p)| =  \infty.$
        We can assume $f(p) \notin E.$
        Since $M_{\alpha}(\mu_{f(p)})(E) > 0,$ there exists an arc $I$ centred at $f(p)$ such that $M_{\alpha}(\mu_{f(p)})(E \setminus I) > 0.$
        Write $E^* = E \setminus I.$
        Then there exists $n_0 >0$ such that $M_{\alpha}(\mu_{f(p_n)})(E^*) > M_{\alpha}(\mu_{f(p)})(E^*) /2$ if $n > n_0.$
        Now,
        \begin{equation*}
            \begin{split}
                M_{\alpha}(\mu_p)(f^{-1}(E^*)) &= \lim_{n\to \infty} M_{\alpha}(\mu_{p_n})(f^{-1}(E^*))\\
                &\geq C_\alpha \limsup_{n\to \infty} \left(\frac{1-|f(p_n)|^2}{1-|p_n|^2}\right)^{\alpha} M_{\alpha}(\mu_{f(p_n)})(E^*) = \infty.
            \end{split}
        \end{equation*}
        Hence $M_{\alpha}(\mu_p)(f^{-1}(E)) = \infty.$
    \end{proof}
    
    \section{Applications}
    \label{sec:Applications}
    \subsection{Omitted values}
    A classical result by Frostman says that any inner function $f$ can omit at most a set of logarithmic capacity zero, that is, $\disk \setminus f(\disk)$ has logarithmic capacity zero (see Chapter II of \cite{ref:GarnettBoundedAnalyticFunctions}).
    Conversely, given a relatively compact set $K$ of the unit disc of logarithmic capacity zero, the universal covering map $f\colon \disk \rightarrow \disk \setminus K$ is an inner function (see page 323 of \cite{ref:TsujiPotentialTheory}).
    Given a set $E \subseteq \disk,$ its \emph{non-tangential closure} on $\ddisk$, denoted by $E^{NT},$ is the set of points $\xi \in \ddisk$ for which  there exists a sequence $\{z_n\} \subseteq E$ such that $z_n \to \xi$ non-tangentially. We first state an auxiliary result which may have independent interest. 
    
    \begin{lemma}
        \label{lemma:inclusion}
        Let $f\colon \disk \rightarrow \disk$ be an inner function and let $E = \disk \setminus f(\disk)$ be the set of its omitted points. Then 
        \begin{equation*}
            f^{-1}(E^{NT}) \subseteq \{\xi \in \ddisk\colon |f'(\xi)| = \infty\}.
        \end{equation*}
    \end{lemma}

    \begin{proof}
        Consider a point $\xi \in \ddisk$ such that the angular derivative of $f$ at $\xi$ exists and it is finite, and let $\zeta = f(\xi).$
        In other words assume that
        \begin{equation}
            \label{eq:FiniteAngularDerivative}
            \lim_{\Gamma_\beta(\xi) \ni z \rightarrow \xi } \frac{\zeta - f(z)}{\xi-z} = A
        \end{equation}
        is finite.
        We want to see that, in this situation, for any opening $\gamma > 1,$ there is $0 < s = s(\gamma) < 1$ such that the truncated cone
        \begin{equation*}
            \Gamma_{\gamma,s}(\zeta) = \left\{w \in \disk\colon |\zeta-w| < \gamma (1-|w|), |\zeta-w| < s \right\}
        \end{equation*}
        does not intersect $E,$ that is, $\Gamma_{\gamma,s}(\zeta) \subseteq f(\disk).$
        So fix $\gamma > 1$ and consider $\Gamma_{\gamma,s}(\zeta)$ with $0 < s < 1$ to be determined.
        Fix $w_0 \in \Gamma_{\gamma,s}(\zeta).$
        We want to see that there is $z_0 \in \disk$ such that $f(z_0) = w_0.$
        By equation \eqref{eq:FiniteAngularDerivative}, we can express
        \begin{equation*}
            f(z) = \zeta + A (z-\xi) + o(|z-\xi|),
        \end{equation*}
        where $o(|z-\xi|)/|z-\xi| \rightarrow 0$ as $z \to \xi$ non-tangentially.
        Consider $\Gamma_{\beta,r}(\xi)$ with $\beta > 2\gamma$ and $0 < r < 1$ to be determined.
        Observe that there exists $0 < r_0 < 1$ such that, if $r < r_0$ and $0< s < |A|r/2,$ then for any $z \in \partial\Gamma_{\beta,r}(\xi)$ we have that
        \begin{equation*}
            \left|(f(z)-w_0) - (\zeta+A(z-\xi)-w_0)\right| < |\zeta+A(z-\xi)-w_0|.
        \end{equation*}
        Thus, by Rouché's Theorem, the functions $f(z)-w_0$ and $g(z)-w_0 = \zeta+A(z-\xi)-w_0$ have the same number of zeroes in $\Gamma_{\beta,r}(\xi).$
        But $g(z)$ is a degree $1$ polynomial and $g(\Gamma_{\beta,r}(\xi)) = \Gamma_{\beta,|A|r}(\zeta) \supseteq \Gamma_{\gamma,s}(\zeta),$ and thus $g(z)-w_0$ has a single zero on $\Gamma_{\beta,r}(\xi).$
        Therefore, there is $z_0 \in \Gamma_{\beta,r}(\xi)$ such that $f(z_0) = w_0,$ which completes the proof.
    \end{proof}
    
    As an application of Theorem \ref{thm:LownerMuP} and Lemma \ref{lemma:inclusion}, we have the following result.
    \begin{corl}
        Let $f\colon \disk \rightarrow \disk$ be an inner function and let $E = \disk \setminus f(\disk)$ be the set of its omitted points.
        Let $p$ be a boundary Fatou point of $f$.
        
        \begin{enumerate}[(a)]
            \item
            Assume $|f'(p)| < \infty.$
            Then for any $0 < \alpha < 1$ there exists a constant $C_\alpha > 0,$ independent of $f,$ such that
            \begin{equation}
                \label{eq:InfiniteDerivativeContent}
                M_\alpha(\mu_p)\left(\{\xi \in \ddisk\colon |f'(\xi)| = \infty\}\right) \geq C_\alpha |f'(p)|^\alpha M_\alpha(\mu_{f(p)})(E^{NT}).
            \end{equation}
            
            \item
            Assume $|f'(p)| = \infty.$
            Then  $M_\alpha(\mu_p)\left(\{\xi \in \ddisk\colon |f'(\xi)| = \infty\}\right) = \infty$ whenever $M_\alpha(\mu_{f(p)})(E^{NT}) > 0.$
        \end{enumerate}
    \end{corl}

    \subsection{Inner functions in the upper half plane}
    Let $\halfpl = \{w \in \mathbb{C}\colon \Im(w) > 0 \}$ be the upper half plane.
    A holomorphic mapping $g\colon \halfpl \rightarrow \halfpl$ is an \emph{inner function} of the upper half plane if $\lim_{y \to 0} g(x + iy) \in \reals$ for a.e. $x \in \reals.$
    This natural definition agrees with conformal changes of coordinates: given $p \in \ddisk$ denote by  $w_p$ the Möbius transformation mapping $\disk$ onto $\halfpl,$ the point $p$ to $\infty$ and, say, the origin to $i/2.$
    Then, $g$ is an inner function of the upper half plane if and only if $f = w_p^{-1} \circ g \circ w_p$ is an inner function of the unit disc $\disk.$
    Observe that $g(\infty) = \lim_{t \to +\infty} g(it) = \infty$ if and only if $f(p) = p.$
    A holomorphic mapping $g$ from $\halfpl$ into $\halfpl$ has a finite angular derivative at $\infty$ if 
    \begin{equation*}
        g'(\infty) = \lim_{t \to +\infty} \frac{it}{g(it)}
    \end{equation*}
    exists and is finite.
    Otherwise, we write $|g'(\infty)| = \infty.$
    Observe that $g$ has a finite angular derivative at infinity if and only if $f = w_p^{-1} \circ g \circ w_p$ has a finite angular derivative at $p.$
    Moreover, the identity $|g'(\infty)| = |f'(p)|$ holds in the sense that both quantities coincide when they are finite, and if one of them is infinite so is the other.
    This fact easily follows from the identity
    \begin{equation*}
        \frac{w}{g(w)} = \frac{p + z}{p + f(z)}\frac{p - f(z)}{p - z}.
    \end{equation*}
    Let $|A|$ denote the Lebesgue measure of a measurable set $A \subseteq \reals$ and, for $0 < \alpha < 1,$ let $M_\alpha (A)$ denote its $\alpha$-Hausdorff content.
    We now state the versions of \eqref{thm:LownerMuP} and \eqref{thm:ContentsMuP} in this setting.  

    \begin{corl}
        \label{cor:LownerMuP}
        Let $g\colon \halfpl \to \halfpl$ be an inner function and assume that  $g(\infty) = \infty.$
        
        \begin{enumerate}[(a)]
            \item
            \label{stm:LownerMuPFiniteDerivativeUpperHalfPlane}
            Assume $|g'(\infty)| < \infty.$
            Then
            \begin{equation}
                \label{first}
                |g^{-1}(A)| = |g'(\infty)| |A|
            \end{equation}
            for any measurable set $A \subseteq \reals.$
            Moreover, for any $0 < \alpha < 1$ there exists a constant $C_\alpha > 0,$ independent of $g,$ such that
            \begin{equation}
                \label{formulacontent}
                M_{\alpha} (g^{-1}(A)) \geq C_{\alpha} |g'(\infty)|^{\alpha} M_{\alpha}(A)
            \end{equation}
            for any Borel set $A \subseteq \reals.$
            
            \item
            \label{stm:LownerMuPInfiniteDerivativeUpperHalfPlane}
            If $|g'(\infty)| = \infty$ and $A \subseteq \reals$ is a measurable set, then $|g^{-1}(A)| = \infty$ if $|A| > 0$ and  $|g^{-1}(A)| = 0$ if $|A|= 0.$
            Moreover, $M_{\alpha}(g^{-1}(A)) = \infty$ for any Borel set $A \subseteq \reals$  such that  $M_{\alpha}(A) > 0.$
        \end{enumerate}
    \end{corl}

    \begin{proof}
        Note that for any measurable set $A \subseteq \reals$ we have \begin{equation}
            \label{leb r}
            |A| = \mu_p(w_p^{-1}(A)), \qquad p \in \ddisk.
        \end{equation}
        Hence, $|g^{-1}(A)| = \mu_p(w_p^{-1}(g^{-1}(A))) = \mu_p(f^{-1} (w_p^{-1}(A))).$
        Applying Theorem \ref{thm:LownerMuP} and \eqref{leb r} we deduce $|g^{-1}(A)| = |f'(p)| \mu_p(w_p^{-1}(A)) = |g'(\infty)| |A|$ which is \eqref{first}.
        It follows from \eqref{leb r} and $w_p$ being a Möbius map that 
        \begin{equation}
            \label{rel contents}
            M_\alpha(\mu_p)(E) = M_\alpha(w_p(E)), \quad E \subseteq \ddisk.
        \end{equation}
        Thus, the previous argument shows that \eqref{formulacontent} holds.
        Part \eqref{stm:LownerMuPInfiniteDerivativeUpperHalfPlane} follows from similar considerations.
    \end{proof}
    
    \printbibliography

    \Addresses

\end{document}